\documentclass[11pt]{amsart}

\usepackage{amscd}
\usepackage{amsmath, amssymb}
\usepackage{amsfonts}

\newcommand{\ddbar}{\sqrt{-1} \partial \overline{\partial}}
\newcommand{\Ric}{\mathrm{Ric}}
\newcommand{\ov}[1]{\overline{#1}}

\newcommand{\tr}[2]{\mathrm{tr}_{#1}{#2}}

\newcommand{\vp}{\varphi}

\newcommand{\ve}{\varepsilon}

\renewcommand{\leq}{\leqslant}
\renewcommand{\geq}{\geqslant}

\numberwithin{equation}{section}

\begin{document}
\newtheorem{claim}{Claim}
\newtheorem{theorem}{Theorem}[section]
\newtheorem{lemma}[theorem]{Lemma}
\newtheorem{corollary}[theorem]{Corollary}
\newtheorem{proposition}[theorem]{Proposition}
\newtheorem{conjecture}[theorem]{Conjecture}

\theoremstyle{definition}
\newtheorem{remark}[theorem]{Remark}
\newtheorem{question}[theorem]{Question}

\author[V. Tosatti]{Valentino Tosatti$^{*}$}
 \address{Department of Mathematics, Northwestern University, 2033 Sheridan Road, Evanston, IL 60208}
\email{tosatti@math.northwestern.edu}
\author[X. Yang]{Xiaokui Yang}
 \address{Academy of Mathematics and Systems Science, The Chinese Academy of Sciences, Beijing, China, 100190}
\email{xkyang@amss.ac.cn}

\title{An extension of a theorem of Wu-Yau}

\thanks{$^{*}$Supported in part by a Sloan Research Fellowship and NSF grant DMS-1308988.}

\begin{abstract}
We show that a compact K\"ahler manifold with nonpositive holomorphic sectional curvature has nef canonical bundle. If the holomorphic sectional curvature
is negative then it follows that the canonical bundle is ample, confirming a conjecture of Yau. The key ingredient is the recent solution of this conjecture in the
projective case by Wu-Yau.
\end{abstract}
\maketitle

\section{Introduction}
The holomorphic sectional curvature of a K\"ahler manifold is by definition equal to the Riemannian sectional curvature of holomorphic planes in the tangent space, and it determines the whole curvature tensor. Despite this simple definition, its properties have remained rather mysterious. In this note, we study its relationship with Ricci curvature, and more precisely how negativity of the holomorphic sectional curvature affects the positivity of the canonical bundle.

Yau's Schwarz Lemma \cite{YaS} implies that if a compact K\"ahler manifold $(X,\omega)$ has nonpositive
holomorphic sectional curvature, then $X$ does not contain any rational curve. If $X$ is assumed to be projective, then thanks to Mori's Cone Theorem \cite{Mo} it follows that $K_X$ is nef (which means that $c_1(K_X)=-c_1(X)$ belongs to the closure of the K\"ahler cone), since if $K_X$ is not nef then $X$ contains a rational curve.
We conclude that if $X$ is projective and  $\omega$ has nonpositive
holomorphic sectional curvature then $K_X$ is nef.

Our goal is to extend this result to all K\"ahler manifolds, not necessarily projective. Note that there are many non-projective K\"ahler manifolds with nonpositive
holomorphic sectional curvature, for example a generic torus.

The Kodaira classification of K\"ahler surfaces shows that if $K_X$ is not nef then again $X$ contains a rational curve, and we are done.
The Cone Theorem was very recently extended to K\"ahler threefolds by H\"oring-Peternell \cite{HP}, and therefore in this case we can again conclude what we want. In the main result of this note, we prove this in all dimensions (bypassing Mori theory), using very recent ideas of Wu-Yau \cite{WY}:

\begin{theorem}\label{main}
Let $(X,\omega)$ be a compact K\"ahler manifold with nonpositive holomorphic sectional curvature. Then the canonical bundle $K_X$ is nef.
\end{theorem}

As a corollary, using the work of Wu-Yau \cite{WY}, we confirm a conjecture of Yau (see e.g. \cite{WY}), which was very recently settled in the projective case by Wu-Yau \cite{WY}:
\begin{corollary}\label{cor}
Let $(X,\omega)$ be a compact K\"ahler manifold with negative holomorphic sectional curvature. Then the canonical bundle $K_X$ is ample.
\end{corollary}

Indeed, Theorem \ref{main} shows that $K_X$ is nef, and we can then apply \cite[Theorem 7]{WY} and conclude that $K_X$ is ample.

This also gives a proof of another open problem posed by Yau \cite[p.1313]{Wo2}:
\begin{corollary}\label{cor2}
Let $(X,\omega)$ be a compact K\"ahler manifold with negative holomorphic sectional curvature. Then $X$ is projective.
\end{corollary}

Using Yau's Theorem \cite{Ya}, we can rephrase Corollary \ref{cor} by saying that if there exists a K\"ahler metric with negative holomorphic sectional curvature, then there is a (possibly different) K\"ahler metric with negative Ricci curvature.

When $X$ is projective, Corollary \ref{cor} was proved very recently by Wu-Yau \cite{WY}, and we make use of some of their ideas in this note. Earlier work on this conjecture include \cite{Wo} (where it is proved for K\"ahler surfaces), \cite{HLW} (for projective threefolds), \cite{WWY} (for projective manifolds with Picard number $1$), \cite{HLW2} (for projective manifolds assuming the abundance conjecture) and finally \cite{WY} (for all projective manifolds). In this note we observe that a modification of the technique in \cite{WY}, using also an argument by contradiction, can be used to prove Theorem \ref{main} as well, and therefore get rid of the projectivity assumption.

\begin{remark}
Corollary \ref{cor} falls into the circle of ideas around Kobayashi hyperbolicity of complex manifolds, see e.g. \cite{Di, Pe, WY} for more on this.
\end{remark}

\begin{remark}As in \cite[Corollary 4]{WY}, Theorem \ref{main} implies that if $(X,\omega)$ be a compact K\"ahler manifold with nonpositive holomorphic sectional curvature,
then the canonical bundle $K_Y$ of every compact complex submanifold $Y$ of $X$ is nef. Similarly, if $\omega$ has negative holomorphic sectional curvature, then $K_Y$ is ample.
\end{remark}

\begin{remark}
It would be interesting to know whether Corollary \ref{cor} still valid if we only assume that the holomorphic sectional curvature is nonpositive and strictly negative at one point. This is shown to be true in \cite{WWY} if $X$ is projective with Picard number $1$, and in \cite{HLW2} for projective surfaces.
\end{remark}

\begin{remark}
If $(X,\omega)$ is a compact Hermitian manifold, then we can still define its holomorphic sectional curvature, using the Chern connection. Is the analog of Theorem \ref{main} still true in this more general situation? While the analogous complex Monge-Amp\`ere equation in the Hermitian case is solved in \cite{Ch} (see also \cite{TW1,TW2}), the argument with the Schwarz Lemma uses the K\"ahler condition for Royden's lemma (which requires the K\"ahler symmetries of the curvature tensor), and also to make sure that all the Hermitian Ricci curvatures are equal.
\end{remark}

\noindent
{\bf Acknowledgments. }We thank Professor S.-T. Yau for useful comments on this note, and for his support, and S. Diverio for communications. We also acknowledge an intellectual debt of gratitude to the work of Wu-Yau.

\section{Proof of the main theorem}

In this section we give the proof of Theorem \ref{main}, by modifying the method used by Wu-Yau \cite{WY} to prove Corollary \ref{cor} in the projective case.

The following is the key lemma, from Wu-Yau \cite{WY} (cf. \cite{WYZ, WWY}), and it is an application of Yau's Schwarz Lemma \cite{YaS} and a trick of Royden \cite{Ro}.

\begin{lemma}[\cite{WY}, Proposition 9]\label{schwarz}
Let $X$ be a compact K\"ahler manifold with two K\"ahler metrics $\omega,\hat{\omega}$, such that $\omega$ has holomorphic sectional curvature bounded above by a constant $-\kappa\leq 0$, and $\hat{\omega}$ satisfies
\begin{equation}\label{assum}
\Ric(\hat{\omega})\geq -\lambda\hat{\omega}+\nu\omega,
\end{equation}
for some constants $\lambda,\nu>0$. Then we have
$$\Delta_{\hat{\omega}}\log\tr{\hat{\omega}}{\omega}\geq\left(\frac{n+1}{2n}\kappa+\frac{\nu}{n}\right)\tr{\hat{\omega}}{\omega}-\lambda.$$
In particular, we have
$$\sup_X \tr{\hat{\omega}}{\omega}\leq \frac{\lambda}{\frac{n+1}{2n}\kappa+\frac{\nu}{n}}.$$
\end{lemma}
\begin{proof}
Yau's Schwarz Lemma calculation \cite{YaS} gives
$$\Delta_{\hat{\omega}}\log\tr{\hat{\omega}}{\omega}\geq \frac{1}{\tr{\hat{\omega}}{\omega}}\left(\hat{g}^{i\ov{\ell}}\hat{g}^{k\ov{j}}g_{k\ov{\ell}}\hat{R}_{i\ov{j}}
-\hat{g}^{i\ov{j}}\hat{g}^{k\ov{\ell}}R_{i\ov{j}k\ov{\ell}}\right).$$
An observation of Royden \cite[Lemma, p.552]{Ro} gives
$$\hat{g}^{i\ov{j}}\hat{g}^{k\ov{\ell}}R_{i\ov{j}k\ov{\ell}}\leq -\frac{n+1}{2n}\kappa (\tr{\hat{\omega}}{\omega})^2,$$
while \eqref{assum} says that
$$\hat{R}_{i\ov{j}}\geq -\lambda \hat{g}_{i\ov{j}}+\nu g_{i\ov{j}},$$
and so
$$\hat{g}^{i\ov{\ell}}\hat{g}^{k\ov{j}}g_{k\ov{\ell}}\hat{R}_{i\ov{j}}\geq -\lambda \tr{\hat{\omega}}{\omega} +\nu\hat{g}^{i\ov{\ell}}\hat{g}^{k\ov{j}}g_{k\ov{\ell}}g_{i\ov{j}}
\geq -\lambda \tr{\hat{\omega}}{\omega} +\frac{\nu}{n}(\tr{\hat{\omega}}{\omega})^2.$$
\end{proof}

We now prove Theorem \ref{main} by contradiction. If $K_X$ is not nef, then there exists $\ve_0>0$ such that the class
$\ve_0[\omega]-c_1(X)$ is nef but not K\"ahler. Then for every $\ve>0$ the class $(\ve+\ve_0)[\omega]-c_1(X)$ is K\"ahler, and so we can find
a smooth function $\vp_\ve$ such that
$$(\ve+\ve_0)\omega-\Ric(\omega)+\ddbar\vp_\ve>0.$$
By a theorem of Yau \cite{Ya} and Aubin \cite{A}, we can find a smooth function $\psi_\ve$ such that
$$(\ve+\ve_0)\omega-\Ric(\omega)+\ddbar(\vp_\ve+\psi_\ve)>0,$$
and
$$((\ve+\ve_0)\omega-\Ric(\omega)+\ddbar(\vp_\ve+\psi_\ve))^n=e^{\vp_\ve+\psi_\ve}\omega^n.$$
We let $u_\ve=\vp_\ve+\psi_\ve$ and $\omega_\ve=(\ve+\ve_0)\omega-\Ric(\omega)+\ddbar(\vp_\ve+\psi_\ve)$, so that we can write
\begin{equation}\label{ma}
\omega_\ve^n=e^{u_\ve}\omega^n.
\end{equation}
Differentiating this, we see that
\begin{equation}\label{ric}
\Ric(\omega_\ve)=\Ric(\omega)-\ddbar u_\ve=-\omega_\ve+(\ve+\ve_0)\omega.
\end{equation}
We may therefore apply Lemma \ref{schwarz}, with $\kappa=0, \lambda=1, \nu=\ve+\ve_0$, and obtain
\begin{equation}\label{uno}
\sup_X \tr{\omega_\ve}{\omega}\leq\frac{n}{\ve+\ve_0},
\end{equation}
which is bounded uniformly independent of $\ve$ (as $\ve$ approaches zero).
Furthermore, at any point $x\in X$ where $u_\ve$ achieves its maximum, we have that $((\ve+\ve_0)\omega-\Ric(\omega))(x)>0$ and
$$e^{\sup_X u_\ve}=e^{u_\ve(x)}\leq \frac{((\ve+\ve_0)\omega-\Ric(\omega))^n}{\omega^n}(x)\leq C,$$
independent of $\ve$ small. This proves a uniform upper bound for $u_\ve$, and hence we obtain
\begin{equation}\label{due}
\sup_X\frac{\omega_\ve^n}{\omega^n}\leq C.
\end{equation}
Combining \eqref{uno}, \eqref{due} and the elementary inequality
$$\tr{\omega}{\omega_\ve}\leq \frac{1}{(n-1)!}(\tr{\omega_\ve}{\omega})^{n-1}\frac{\omega_\ve^n}{\omega^n},$$
we conclude that
\begin{equation}\label{tre}
\sup_X\tr{\omega}{\omega_\ve}\leq C,
\end{equation}
and \eqref{uno} and \eqref{tre} together give
\begin{equation}\label{quattro}
C^{-1}\omega\leq \omega_\ve\leq C\omega.
\end{equation}
Using \eqref{ma} this implies that $\inf_Xu_\ve\geq -C$ as well.
We claim that the following higher order estimates hold
\begin{equation}\label{cinque}
\|\omega_\ve\|_{C^k(X,\omega)}\leq C_k,
\end{equation}
where $C_k$ is independent of $\ve$, for all $k\geq 0$. These are essentially standard (following the work of Yau \cite{Ya}), but since the precise setting
is not readily found in the literature (we have no control on our reference metrics $(\ve+\ve_0)\omega-\Ric(\omega)+\ddbar\vp_\ve$ as $\ve\to 0$), we  provide a proof below.

Assuming first that we have \eqref{cinque}, we can immediately conclude the proof of Theorem \ref{main}. Indeed, using \eqref{cinque} together with \eqref{quattro}, with the Ascoli-Arzel\`a theorem and a diagonal argument, we obtain that there exists a sequence $\ve_i\to 0$ such that $\omega_{\ve_i}$ converge smoothly to a K\"ahler metric $\omega_0$ which satisfies
$$[\omega_0]=\ve_0[\omega]-c_1(X),$$
which is a contradiction to the fact that this class is not K\"ahler.

For the reader's convenience, we now give the proof of the higher-order estimates \eqref{cinque}, following \cite{Ya}. Let
$$S=|\nabla^\omega \omega_\ve|^2_{\omega_\ve},$$
where $\nabla^\omega$ is the covariant derivative of $\omega$. It is also equal to $S=|T|^2_{\omega_\ve}$, where $T$ is the tensor which is
given by the difference of the Christoffel symbols of $\omega_\ve$ and $\omega$. Yau's $C^3$ calculation gives (cfr. \cite[(2.44)]{PSS})
\[\begin{split}
\Delta_{\omega_\ve} S &=   |\nabla T|^2_{\omega_\ve} +|\ov{\nabla} T|^2_{\omega_\ve}-2\mathrm{Re}(g_\ve^{i\ov{p}}g_\ve^{j\ov{q}}(g_\ve)_{k\ov{\ell}}\ov{T^\ell_{pq}}(\nabla_i R^k_j-
\nabla^{\ov{\ell}}R(\omega)_{i\ov{\ell}j}^{\ \ \ k})  )\\
&+T_{ij}^k\ov{T_{pq}^\ell}(R^{i\ov{p}}g_\ve^{j\ov{q}}(g_\ve)_{k\ov{\ell}}+g_\ve^{i\ov{p}}R^{j\ov{q}}(g_\ve)_{k\ov{\ell}}-g_\ve^{i\ov{p}}g_\ve^{j\ov{q}}R_{k\ov{\ell}}   ),
\end{split}\]
where $\nabla$ denotes the covariant derivative of $\omega_\ve$, $R_{k\ov{\ell}}$ its Ricci curvature, $R(\omega)$ is the curvature tensor of $\omega$, and we raise indices using $\omega_\ve$.
Using \eqref{ric} together with \eqref{quattro} we obtain
$$\Delta_{\omega_\ve} S\geq -C_0S-C.$$
But we also have
\[\begin{split}\Delta_{\omega_\ve}\tr{\omega}{\omega_\ve}&=\Delta_{\omega}u_\ve -R_\omega+g^{i\ov{\ell}}\, g_\ve^{p\ov{j}} g_\ve^{k\ov{q}} \nabla^\omega_i (g_\ve)_{k\ov{j}} \nabla^\omega_{\ov{\ell}}(g_\ve)_{p\ov{q}}+g^{i\ov{\ell}}\,g_\ve^{p\ov{j}} (g_\ve)_{k\ov{\ell}}R(\omega)_{p\ov{j}i}^{\ \ \ k}\\
&\geq C_1^{-1}S-C,
\end{split}\]
where $R_\omega$ is the scalar curvature of $\omega$,
and so the maximum principle applied to $S+C_1(C_0+1)\tr{\omega}{\omega_\ve}$ gives $\sup_X S\leq C$, independent of $\ve$ small. This implies that
$\|\omega_\ve\|_{C^1(X,\omega)}\leq C$, and then a standard bootstrap argument gives all the higher order estimates \eqref{cinque}.

\end{document}